\documentclass[a4paper,reqno]{amsart} 

\usepackage{amssymb}
\usepackage[mathcal]{euscript} 

\usepackage{tikz-cd} 
\usetikzlibrary{babel}

\usepackage{hyperref}

\newcommand{\bydef}{:=} 
\newcommand{\defby}{=:}

\newcommand{\id}{\mathrm{id}}

\newcommand{\trace}{\mathrm{tr}}

\newcommand{\cA}{\mathcal{A}}

\newcommand{\cL}{\mathcal{L}} 
\newcommand{\cM}{\mathcal{M}}

\newcommand{\ZZ}{\mathbb{Z}}

\newcommand{\RR}{\mathbb{R}}

\newcommand{\FF}{\mathbb{F}} 

\DeclareMathOperator{\rad}{\mathrm{rad}}

\DeclareMathOperator{\End}{\mathrm{End}}
\DeclareMathOperator{\Isom}{\mathrm{Isom}}

\DeclareMathOperator{\AAut}{\mathbf{Aut}}

\DeclareMathOperator{\alg}{\mathrm{alg}}

\newcommand{\frgl}{{\mathfrak{gl}}}

\newcommand{\Ort}{\mathrm{O}}

\newcommand{\SP}{\mathrm{Sp}}

\newenvironment{alphaenumerate} 
{\begin{enumerate}

}{\end{enumerate}}

\newcommand{\subo}{_{\bar 0}} 
\newcommand{\subuno}{_{\bar 1}}

\newcommand{\bup}{\textup{b}}

\newtheorem{theorem}{Theorem}[section]
\newtheorem{proposition}[theorem]{Proposition}

\newtheorem{corollary}[theorem]{Corollary}

\theoremstyle{definition} 
\newtheorem{definition}[theorem]{Definition}
\newtheorem{example}[theorem]{Example}

\theoremstyle{remark} \newtheorem{remark}[theorem]{Remark}
\numberwithin{equation}{section}

\begin{document}

\title
{Vidinli algebras}

\author[A.~Elduque]{Alberto Elduque} 
\address{Departamento de
Matem\'{a}ticas e Instituto Universitario de Matem\'aticas y
Aplicaciones, Universidad de Zaragoza, 50009 Zaragoza, Spain}
\email{elduque@unizar.es} 
\thanks{Both authors have been supported by grant
PID2021-123461NB-C21, funded by 
MCIN/AEI/ 10.13039/501100011033 and by
 ``ERDF A way of making Europe''. A.E. also acknowledges support by grant 
E22\_20R (Gobierno de Arag\'on), while J.R-I. acknowledges support by
grant Fortalece 2023/03 funded by ``Comunidad Aut\'onoma 
de La Rioja''; and by a predoctoral research grant FPI-2023 funded by ``Universidad de La Rioja''.}

\author[J.~R\'andez-Ib\'a\~nez]{Javier R\'andez-Ib\'a\~nez}
\address{Departamento de Matem\'aticas y Computaci\'on,
Universidad de La Rioja,\phantom{26006 26006} 26006 Logro{\~n}o, Spain}
\email{javier.randez@unirioja.es}




\begin{abstract}
A new class of nonassociative algebras, Vidinli algebras, is defined based on recent work of 
Co\c skun and Eden. These algebras are conic (or quadratic) algebras with the extra restriction that
the commutator of any two elements is a scalar multiple of the unity.

Over fields of characteristic not $2$, Vidinli algebras may be considered as generalizations of
the Jordan algebras of Clifford type. However, in characteristic $2$, the class of Vidinli algebras is
much larger and include the unitizations of anticommutative algebras.
\end{abstract}

\maketitle

\bigskip

\section{Introduction}

In a recent preprint: \cite{CE}, a nonassociative algebra defined long ago in \cite{Vidinli} by H\"useyin
Tevfik Pasha, also named Vidinli because of his geographical origin, is reviewed and extended. The
algebras considered in \cite{CE}, called \emph{Vidinli algebras}, 
are the algebras defined on $\RR^{2n+1}$, endowed with its
canonical scalar product $x\cdot y$, and with multiplication given by
\[
xy=(x\cdot 1)y+(y\cdot 1)1-(x\cdot y)1+\omega(x,y)1,
\]
where $1$ is a fixed element of length one and $\omega$ is a suitable nondegenerate skew-symmetric
bilinear form defined on the subspace orthogonal to $1$ and  extended to the 
whole vector space by imposing $\omega(1,z)=0$ for all $z$. The algebra defined in \cite{Vidinli}
is the one where $n=1$.

Note that the element $1$ is the unity of this algebra: $1x=x=x1$ for all $x$, and the multiplication
of elements orthogonal to $1$ is just given by: 
\[
xy=\bigl(-(x\cdot y)+\omega(x,y)\bigr)1.
\]

\medskip

The goal of this paper is to define a natural class of algebras over arbitrary fields, that will also be
called \emph{Vidinli algebras}, that include the algebras defined above, and to
study their properties.
In this way, we may look at the algebras considered in \cite{CE} with a suitable perspective, that
highlights their basic features, and relate them to other well-known classes of nonassociative algebras.

\smallskip

Over fields of characteristic not $2$, Vidinli algebras are close to the Jordan algebras of symmetric 
bilinear forms,
also termed spin factors or Jordan algebras of Clifford type (see \cite{McCrimmon}). However,
in characteristic $2$, the class of Vidinli algebras is very large and includes the unitizations of 
the anticommutative algebras (Example \ref{ex:char2}).

\medskip

In Section \ref{se:Vidinli}, the Vidinli algebras will be defined (Definition \ref{df:Vidinli}) as those 
conic algebras satisfying that $[x,y]=xy-yx$ is a scalar multiple of the unity for all $x,y$. Conic algebras
are those unital algebras that satisfy a sort of Cayley-Hamilton equation of degree $2$ (see
Section \ref{se:Vidinli} for details).

Once Vidinli algebras are defined, some important examples will be provided.

\smallskip

Section \ref{se:charnot2} will be devoted to studying Vidinli algebras over a field $\FF$ 
of characteristic not $2$. In this case, any Vidinli algebra is endowed with an 
associated bilinear form $B$, and $A$
decomposes as $A=\FF 1\oplus V$, with multiplication given by \eqref{eq:aubv}:
\[
(\alpha 1+u)(\beta 1+v)=\bigl(\alpha\beta-B(u,v)\bigr)1+(\alpha v+\beta u),
\]
for any $\alpha,\beta\in \FF$ and $u,v\in V$. This is exactly the formula for the multiplication in a 
Jordan algebra of a symmetric bilinear form, except for the fact that $B$ is not necessarily symmetric. 
Therefore, Vidinli algebras may be considered as generalizations of these Jordan algebras.

\smallskip

It then turns out that the radical $\rad B=\{x\in A\mid B(x,A)=0=B(A,x)\}$ is an ideal of $A$ that
squares to $0$, and that $A/\!\rad B$ is either simple or isomoprhic to $\FF\times\FF$ 
(Theorem \ref{th:rad}).

The associated bilinear form $B$ determines the Vidinli algebra over fields of characteristic not $2$,
and the determination of the group of automorphisms, Lie algebra of derivations, multiplication algebra,
center, ..., performed in Section \ref{se:charnot2}, depends naturally on $B$.

\smallskip

The final Section \ref{se:char2} deals with Vidinli algebras over a field $\FF$ of characteristic $2$. 
Here, and assuming that the dimension is at least $3$, any Vidinli algebra splits, as before, 
as $A=\FF 1\oplus V$, but not canonically!, and the multiplication is given by \eqref{eq:AV*phi}:
\[
(\alpha 1+u)(\beta 1+v)=\bigl(\alpha\beta+\varphi(u,v)\bigr)1+(\alpha v+\beta u+ u*v),
\]
for $\alpha,\beta\in\FF$ and $u,v\in V$, where $\varphi$ is a bilinear form and $*$ an anticommutative
multiplication in $V$. In particular, as in Example \ref{ex:char2}, the unitization of any 
anticommutative algebra, which corresponds to the case $\varphi=0$, is a Vidinli algebra.

\bigskip

In what follows $A$ will always denote a  nonassociative algebra
 over an arbitrary ground field $\FF$. Multiplication will be usually denoted by juxtaposition.
As usual, the bracket $[x,y]$ of two elements in $A$ is given by $[x,y]=xy-yx$. The word \emph{algebra}
will mean nonassociative (i.e., not necessarily associative) algebra, that is, a vector space $A$ endowed
with a bilinear map (the multiplication) $A\times A\rightarrow A$.


\bigskip

\section{Vidinli algebras. Definition and examples}\label{se:Vidinli}

In order to define Vidinli algebras, let us first recall the notion of \emph{conic}
algebra, as in \cite{GP11} and \cite{Loos11} (see also, \cite[\S 16]{GPR25}). A \emph{conic}
algebra is a unital nonassociative algebra $A$ endowed with a quadratic form $q\colon A\rightarrow \FF$,
satisfying the following two conditions:
\begin{equation}\label{eq:conic}
q(1)=1,\quad x^2-q(x,1)x+q(x)1=0\ \forall x\in A,
\end{equation}
where $q(x,y)=q(x+y)-q(x)-q(y)$ is the polar form of $q$, which is denoted by the same letter.
In the literature, these algebras have also been termed quadratic algebras (see, e.g., 
\cite{Osborn,QuadraticAlternative}). They include the classical composition algebras and, in particular,
the quaternion and octonion (or Cayley) algebras. (See \cite{ElduqueSurvey} for a survey on 
composition algebras over fields.)

The quadratic form $q$, which is part of the definition of a conic algebra, is called the \emph{norm}.

Any conic algebra is power-associative, as the subalgebra generated by any element is a commutative
associative algebra of dimension at most $2$.

\medskip

Vidinli algebras are defined as follows:

\begin{definition}\label{df:Vidinli}
A \emph{Vidinli algebra} is a conic algebra $A$ with the added condition that the bracket of any
two elements is a scalar multiple of the unity: $[A,A]\subseteq \FF 1$.
\end{definition}

Let us consider now some basic examples, already mentioned in the Introduction.

\begin{example}\label{ex:Jordan}
Let $J$ be a Jordan algebra of a symmetric bilinear form or, more generally, 
a Jordan superalgebra of a supersymmetric bilinear form,
over a field $\FF$ of characteristic not $2$; that is, $J=\FF 1\oplus V$, where $V=V\subo\oplus V\subuno$
is a vector superspace (i.e., a vector space graded over the integers modulo $2$) endowed with a 
supersymmetric bilinear form $\bup\colon V\times V\rightarrow \FF$, such that $1$ is the unity, 
and $xy=\bup(x,y)1$
for any $x,y\in V$. The fact that $\bup$ is supersymmetric means that $V\subo$ and $V\subuno$ are
orthogonal relative to $\bup$: 
$\bup(V\subo,V\subuno)=0=\bup(V\subuno,V\subo)$, and that the restriction of
$\bup$ to $V\subo$ is symmetric, while its restriction to $V\subuno$ is skew-symmetric.

One checks easily that $J$ is a conic algebra with norm 
\[
q(\alpha 1+x\subo+x\subuno)=\alpha^2-\bup(x\subo,x\subo),
\] 
for $\alpha\in\FF$, $x\subo\in V\subo$,
$x\subuno\in V\subuno$, and that 
\[
[\alpha 1+x\subo+x\subuno,\beta 1+y\subo+y\subuno]=2\bup(x\subuno,y\subuno)1,
\] 
for 
$\alpha,\beta\in \FF$, $x\subo,y\subo\in V\subo$, and $x\subuno,y\subuno\in V\subuno$.
\end{example}

Recall that if $(V,*)$ is an algebra, its \emph{unitization} is the algebra $A=\FF 1\oplus V$, with
multiplication given by $(\alpha 1+u)(\beta 1+v)=\alpha\beta 1+(\alpha v+\beta u+u*v)$,
for $\alpha,\beta\in\FF$ and $u,v\in V$.

\begin{example}\label{ex:char2}
Let $(V,*)$ be an arbitrary anticommutative (i.e., $v*v=0$ for any $v\in V$) algebra over a field of
 characteristic $2$. Then its unitization $A=\FF 1\oplus V$ is a Vidinli algebra with norm given
by $q(\alpha 1+v)=\alpha^2$ for any $\alpha\in \FF$ and $v\in V$.
\end{example}


\bigskip

\section{Vidinli algebras over fields of characteristic not two}\label{se:charnot2}

Assume throughout this section that the characteristic of our ground field $\FF$ is not $2$.

Let $A$ be a Vidinli algebra with norm $q$, and define the skew-symmetric bilinear form 
$\omega\colon A\times A\rightarrow \FF$ by the formula 
\begin{equation}\label{eq:omega}
[x,y]=\omega(x,y)1,
\end{equation}
for $x,y\in A$. Note that we have $\omega(1,A)=0$.

\medskip

\subsection{Associated bilinear form}\label{ss:bilinear}\quad\null

As $q(1)=1$, we get $q(1,1)=2\neq 0$, so that $A$ splits as $A=\FF 1\oplus V$, where
$V=(\FF 1)^\perp\bydef \{x\in A\mid q(x,1)=0\}$. For any $u,v\in V$, the linearization of \eqref{eq:conic}
gives $uv+vu=-q(u,v)1$, while we also have $uv-vu=[u,v]=\omega(u,v)1$. Adding the last two
equations gives
\begin{equation}\label{eq:uv}
uv=\frac{1}{2}(-q(u,v)+\omega(u,v))1,
\end{equation}
for any $u,v\in V$. Define the bilinear form $B\colon A\times A\rightarrow \FF$ by
\begin{equation}\label{eq:B}
B(x,y)=\frac{1}{2}\left(q(x,y)-\omega (x,y)\right)1,
\end{equation}
for $x,y\in A$. The bilinear form $B$ satisfies
\[
B(x,x)=q(x),\quad B(1,x)=B(x,1)=\frac{1}{2}q(x,1),
\]
for any $x\in A$. We will refer to $B$ as the \emph{bilinear form of the Vidinli algebra}.

We summarize the computations above in the first part of the next result, that characterizes the 
Vidinli algebras over fields of characteristic not $2$.

\begin{theorem}\label{th:charnot2}
Let $A$ be a Vidinli algebra over a field of characteristic not $2$, and let $V$ be the orthogonal 
subspace to $\FF1$ relative to the norm of $A$. Let $B\colon A\times A\rightarrow \FF$ be the bilinear
form in \eqref{eq:B}, then the multiplication in $A$ is determined by the
formula:
\begin{equation}\label{eq:aubv}
(\alpha 1+u)(\beta 1+v)=\bigl(\alpha\beta-B(u,v)\bigr)1+\bigl(\alpha v+\beta u\bigr),
\end{equation}
for any $\alpha,\beta\in \FF$ and $u,v\in V$.

Conversely, given any vector space $V$ endowed with a bilinear form $B\colon V\times V\rightarrow\FF$,
the vector space $\FF 1\oplus V$ becomes a Vidinli algebra with the multiplication in \eqref{eq:aubv}
and the norm given by $q(\alpha 1+v)=\alpha^2+B(v,v)$, for any $\alpha\in \FF$ and $v\in V$.
\end{theorem}
\begin{proof} 
The last part follows from straightforward computations.
\end{proof}

\begin{remark}\label{re:Vidinli2}
Theorem \ref{th:charnot2} shows that Vidinli algebras over fields of characteristic not $2$ are determined
by vector spaces endowed with a bilinear form. So they are just like the Jordan algebras of a symmetric 
bilinear form (see Example \ref{ex:Jordan}), but where the symmetry of the bilinear form is not required.

Once we are given a vector space $V$ with a bilinear form $B\colon V\times V\rightarrow \FF$, we construct the 
Vidinli algebra $A=\FF 1\oplus V$ with multiplication given by \eqref{eq:aubv}. The bilinear
form of this Vidinli algebra is obtained by extending $B$ to the whole $A$ by means of 
$B(1,1)=1$ and $B(1,V)=0=B(V,1)$.
\end{remark}

The next consequence of Theorem \ref{th:charnot2} follows at once from \eqref{eq:aubv}, together 
with the fact that $[A,A]$ is contained in $\FF 1$.

\begin{corollary}\label{co:JordanHeisenberg}
Let $A$ be a Vidinli algebra over a field of characteristic not $2$. Then the symmetrized algebra
$A^+$, defined on the same vector space but with new multiplication given by
\[
x\cdot y\bydef \frac{1}{2}\bigl(xy+yx\bigr)
\]
is the Jordan algebra of a symmetric bilinear form (see Example \ref{ex:Jordan}).

On the other hand, the anticommutative algebra $A^-$, that is, the algebra defined on $A$ but with
multiplication given by the bracket: $[x,y]=xy-yx$, is nilpotent of degree $\leq 2$.
\end{corollary}

\medskip

The structure of Vidinli algebras over fields of characteristic not $2$ is similar to the situation
for the Jordan algebras of a symmetric bilinear form, which are a particular case 
(see Example \ref{ex:Jordan}).

\begin{theorem}\label{th:rad}
Let $A$ be a Vidinli algebra over a field $\FF$ of characteristic not $2$, let $B$ be its bilinear form.
Then the following assertions hold:
\begin{alphaenumerate}
\item 
The radical of $B$, defined as $\rad B\bydef\{x\in A\mid B(x,A)=0=B(A,x)\}$, is an ideal of $A$
that satisfies $\bigl(\rad B\bigr)^2=0$.

\item 
$\rad B$ is the largest solvable ideal of $A$, and either $ A/\!\rad B$ is simple, or $A/\!\rad B$ is isomorphic
to $\FF\times\FF$.

\item There is a unital subalgebra $S$ of $A$ such that $A=S\oplus\rad B$.
\end{alphaenumerate}
\end{theorem}
\begin{proof}
Note first that $\rad B$ is contained in $V$, the orthogonal complement of $\FF 1$ relative to the norm.
Equation \eqref{eq:aubv} shows that $V(\rad B)=(\rad B)V=0$, and hence $\rad B$ is an ideal of $A$ and
$\bigl(\rad B\bigr)^2=0$. 

Now, if $I$ is a solvable ideal of $A$, for any $0\neq x\in I$, \eqref{eq:conic} gives  $q(x)=0$, as otherwise
$I$ would contain $1$, and $q(x,1)=0$, as otherwise $I$ would contain an idempotent. Hence $I$
is contained in $V$, and \eqref{eq:aubv} shows that $B(I,V)=0=B(V,I)$. We conclude that $I$ is 
contained in $\rad B$.

Consider the Vidinli algebra $ A/\!\rad B$ or,
equivalently,  assume $\rad B=0$. In this case, if $A$ is not simple and 
$I$ is a proper ideal of $A$, any $u\in I\cap V$
satisfies $uV+Vu=\bigl(B(u,V)+B(V,u)\bigr)1\subseteq I$, but $1\not\in I$, so we get $u\in\rad B=0$.
Hence $I\cap V=0$, and $I$ is one-dimensional, spanned by an element of the form $1+u$, for some
 $0\neq u\in V$. Then, for any $v\in V$, $(1+u)v=-B(u,v)1+v\in I$, which forces $v\in \FF u$ and the
dimension of $A$ is $2$. Moreover, $(1+u)u=-B(u,u)1+u\in I$, forcing $B(u,u)=-1$. We conclude that
$A=\FF 1+\FF u$, and $u^2=1$. This shows that $A$ is isomorphic to $\FF\times\FF$.

Finally, let $S'$ be a subspace of $V$ complementary to $\rad B$: $V=S'\oplus \rad B$, 
then $S=\FF 1\oplus S'$ is a subalgebra
of $A$ complementing $\rad B$.
\end{proof}

\begin{corollary}\label{co:rad}
Unless $A/\!\rad B$ is isomorphic to $\FF\times\FF$, $\rad B$ is the unique maximal ideal of $A$.
\end{corollary}
\begin{proof}
With our hypotheses, if $I$ is an ideal of $A$ not contained in $\rad B$, then $A=I+\rad B$, so $1=a+x$ 
with $a\in I$ and $x\in \rad B$. Then, as $x^2=0$, $1=(1-x)(1+x)=a(1+x)\in I$, so $I=A$, as it contains
the unity.
\end{proof}

\begin{remark}\label{re:rad} If $A/\!\rad B$ is isomorphic to $\FF\times\FF$ and $I$ is a maximal ideal 
of $A$, it must contain $\rad B$, as
otherwise we would have $A=I+\rad B$ and the argument in the proof
of Corollary \ref{co:rad} would give $1\in I$, a contradiction. Hence $\rad B$ is contained in any maximal 
ideal and it follows that $A$ contains exactly two maximal ideals.
\end{remark}

\medskip

\subsection{Vidinli algebras with nondegenerate norm}\label{ss:nondegenerate}\quad\null

If the norm of a finite-dimensional Vidinli algebra $A$ is nondegenerate, 
so is its restriction to $V=(\FF 1)^\perp$, and 
hence there is a linear endomorphism $\sigma\in\End_\FF(V)$ such that 
$\omega(x,y)=q\bigl(\sigma(x),y\bigr)$, for any $x,y\in V$. The skew-symmetry of $\omega$ forces that
$\sigma$ is skew symmetric relative to $q$: $q\bigl(\sigma(x),y\bigr)+q\bigl(x,\sigma(y)\bigr)=0$. By 
Theorem \ref{th:charnot2} the Vidinli algebra $A$ is determined by its norm and the skew
endomorphism $\sigma$.
 
Then the vector space $V$ decomposes  as 
the direct sum of subspaces $V_i$, where the minimal
polynomial of $\sigma\vert_{V_i}$ is a power of a monic irreducible polynomial $p_i(X)\in \FF[X]$.
Proposition 2.1 in \cite{BRR} proves the next result.

\begin{proposition}\label{pr:Vsigma}
Let $A$ be a finite-dimensional Vidinli algebra with nondegenerate norm over a field of 
characteristic not $2$ and let $\sigma\in\End_\FF(V)$ be the skew endomorphism such that
$\omega(x,y)=q\bigl(\sigma(x),y\bigr)$ for any $x,y\in V$. Let $p_1(X),\ldots,p_n(X)\in\FF[X]$ be
the different monic irreducible factors of the characteristic polynomial of $\sigma$. Then for
any $i=1,\ldots,n$ there is a unique $j=1,\ldots,n$ such that $p_j(X)=(-1)^{\deg p_i(X)})p_i(-X)$ ($j$
may be equal to $i$).
Write $j=\hat\imath$.

Reorder the indices so that $n=r+2s$ with  $\hat\imath=i$, for $i=1,\ldots,r$, and $\hat \imath=i+1$, for 
$i=r+2j-1$, $j=1,\ldots,s$. Then $V_i$ is totally isotropic (relative to $q$) for $i=r+1,\ldots r+2s$,
and $V$ is the orthogonal  direct sum (relative to $q$) of the subspaces $V_i$, for $i=1,\ldots,r$, and
$V_i\oplus V_{i+1}$ for $i=r+2j-1$ and $j=1,\ldots,s$.

In particular, $A$ is the sum of the subalgebras $\FF 1\oplus V_i$, for $i=1,\ldots,r$, 
and $\FF 1\oplus(V_i\oplus V_{i+1})$ for $i=r+2j-1$ and $j=1,\ldots,s$. The norm of each of these
subalgebras is nondegenerate.
\end{proposition}

\begin{remark}\label{re:Vsigma} The Vidinli algebras $A$ considered in \cite{CE}, which are 
defined over the real numbers, are finite-dimensional, with a positive definite norm $q$, 
and with the minimal polynomial of 
$\sigma$ equal to $X^2+1$.

In this case it follows at once that $V$ is an orthogonal direct sum of two-dimensional subspaces $W_i$
invariant under $\sigma$, and hence $A$ is the sum of the three-dimensional subalgebras 
$\FF 1\oplus W_i$, which are isomorphic to the original Vidinli algebra defined in \cite{Vidinli}.

A \emph{degenerate pushout} has been defined in \cite{CE} to express the situation above, but
the definition is not correct, as a quotient is taken in \cite[Definition 7.1]{CE} by the span of
elements of the form $ab$ with $a\in A_i\setminus\FF 1$, $b\in A_j\setminus \FF 1$. But the
subset $A_i\setminus \FF 1$ spans the whole $A_i$, and the same for $A_j$, so that in the quotient
all products $ab$, for $a\in A_i$ and $b\in A_j$ become trivial, and the whole definition gives the
trivial algebra.
\end{remark}

\medskip

\subsection{Automorphisms and derivations}\label{ss:autos_der}\quad\null

Any automorphism $\varphi$ of a Vidinli algebra $A=\FF 1\oplus V$ preserves the unity $1$, 
and also the norm $q$, 
because of \eqref{eq:conic}, the skew-symmetric form $\omega$, because of \eqref{eq:omega}, and
the subspace $V=\{x\in A\setminus \FF 1\mid x^2\in\FF 1\}$. Therefore, the restriction of $\varphi$ to
$V$ is an isometry of $q\vert_V$ and $\omega\vert_V$, and hence of $B\vert_V$:
\[
\varphi\vert_V \in \Isom(V,B\vert_V)=\Ort(V,q\vert_V)\cap \SP(V,\omega\vert_V),
\]
where $\Ort(V,q\vert_V)$ (respectively $\SP(V,\omega\vert_V)$) denotes the orthogonal group of 
$q\vert_V$ (resp. the symplectic group associated to $\omega\vert_V$). But conversely,
Equation \eqref{eq:aubv} shows that any linear map preserving $1$ and $V$ and such that
the restriction to $V$ is an isometry of $B\vert_V$ is an automorphism of $A$.

Moreover, all this is functorial, thus obtaining the next result, that extends \cite[Theorem 4.5]{CE}:

\begin{theorem}\label{th:autos}
The group scheme of automorphisms of a Vidinli algebra $A$ over a field $\FF$ of characteristic
not $2$ is isomorphic (by restriction to $V$) to the group scheme of isometries of $B\vert_V$.
\end{theorem}

The Lie algebra of the group scheme of automorphisms of a Vidinli algebra $A$ 
is the Lie algebra of derivations
of $A$, and hence we obtain that this Lie algebra is isomorphic, 
by restriction to $V$, to the Lie algebra of skew endomorphisms relative to $B$.

\begin{theorem}\label{th:der}
The Lie algebra of derivations of a Vidinli algebra $A$ over a field $\FF$ of characteristic not $2$ is 
isomorphic (by restriction to $V$) to the Lie algebra of skew endomorphisms relative to $B\vert_V$:
\[
\{ \delta\in\End_\FF(V)\mid B\bigl(\delta(x),y\bigr)+B\bigl(x,\delta(y)\bigr)=0\ \forall x,y\in V\}.
\]
\end{theorem}
\begin{proof}
Let us provide a self-contained proof.
Any derivation $\delta$ of $A$
annihilates the unity. It also preserves $V$, because for any $x\in V$, $x^2\in\FF 1$, so 
$0=\delta(x^2)=\delta(x)x+x\delta(x)$ which, by the linearization of \eqref{eq:conic} gives
$0=q\bigl(\delta(x),1\bigr)x-q\bigl(x,\delta(x)\bigr)1$, and hence $q\bigl(\delta(x),1\bigr)=0$ and
$\delta(x)\in V$. Now, for $x,y\in V$, $xy$ lies in $\FF 1$, so 
$0=\delta(xy)=\delta(x)y+x\delta(y)=-\Bigl(B\bigl(\delta(x),y\bigr)+B\bigl(x,\delta(y)\bigr)\Bigr)1$, 
for any $x,y\in V$. The converse is clear.
\end{proof}

\medskip

\subsection{Center and nucleus}\label{ss:nucleus}

The \emph{commutative center} of an algebra is the subspace consisting of the elements that commute 
with all the elements of the algebra: 
\[
K(A)\bydef\{x\in A\mid xy=yx\ \forall y\in A\},
\]
(see, e.g., \cite[Chapter II]{Schafer}). Its (associative)
\emph{nucleus} is the subspace of the elements that associate with any other elements, that is, if 
$(x,y,z)\bydef(xy)z-x(yz)$ is the associator of the elements $x,y,z$, the nucleus is the subspace
\[
N(A)\bydef\{x\in A\mid (x,y,z)=(z,x,y)=(y,z,x)=0\ \forall y,z\in A\}.
\]
Finally, the \emph{center} is the subspace (actually subalgebra) 
\[
Z(A)\bydef N(A)\cap K(A).
\]

An element $x=\alpha 1+u$ in a Vidinli algebra $A$, with $u\in V$, is in the commutative center if and only
if $u$ is in the commutative center, and hence, because of \eqref{eq:aubv}, if and only if $B(u,v)=B(v,u)$
for any $v\in V$. Thus, the commutative center of $A$ is the subspace
\[
K(A)=\{ x\in A\mid B(x,y)=B(y,x)\ \forall y\in A\}.
\]
Now, for $u,v\in V$, the associator $(u,v,u)=(uv)u-u(vu)$ equals $\bigl(-B(u,v)+B(v,u)\bigr)u$, and hence
any element $u\in V\cap N(A)$ lies in the commutative center. Thus, $N(A)$ is contained in $K(A)$ and,
therefore, the nucleus coincides with the center.

For $u\in N(A)\cap V$ and any $v,w\in V$, $0=(u,v,w)=-B(u,v)w+B(v,w)u$. Hence, if $B\vert_V\neq 0$, it
follows that $u$ belongs to $\FF w$ for any $w\in V$ with $B(A,w)\neq 0$, and this forces $\dim V=1$, so
that $\dim A=2$, which implies that $A$ is commutative and associative. On the other hand, 
if $B\vert_V=0$, then $A$ is just the unitization of the trivial algebra $V$, and this is commutative and
associative. We have proved the following result:

\begin{proposition}\label{pr:center}
Let $A$ be a Vidinli algebra over a field of characteristic not $2$ and 
let $B$ be its associated bilinear form, then either 
\begin{itemize}
\item $B\vert_V=0$, or
\item $\dim A=2$, or
\item $Z(A)=N(A)=\FF 1$.
\end{itemize} 
In the first two cases, $A$ is commutative and associative, so its center and nucleus are the whole algebra.
\end{proposition}

Recall that an algebra is said to be flexible if and only if $(xy)x=x(yx)$ for any $x,y$.

\begin{corollary}\label{co:center}
Let $A$ be a Vidinli algebra over a field of characteristic not $2$ and 
let $B$ be its associated bilinear form, then the following conditions are 
equivalent:
\begin{alphaenumerate}
\item
$B$ is symmetric,

\item
$A$ is a Jordan algebra,

\item
$A$ is commutative,

\item
$A$ is flexible.
\end{alphaenumerate}
\end{corollary}
\begin{proof}
The implications $\text{(a)}\Rightarrow\text{(b)}\Rightarrow\text{(c)}\Rightarrow\text{(d)}$ are clear.
And if $A$ is flexible, for any $u,v\in V$ we have, as above:
\[
0=(u,v,u)=-B(u,v)u+B(v,u)u,
\]
thus proving that $B$ is symmetric.
\end{proof}

Actually, any flexible conic algebra satisfies that its associated bilinear form is symmetric (see \cite[p.~203]{Osborn}).

\medskip

\subsection{Multiplication algebra}\label{ss:mult_algebra}

Given an algebra $A$ over a field $\FF$, its \emph{multiplication algebra} (see, e.g., 
\cite[Chapter II]{Schafer}) $\cM(A)$  is the associative subalgebra of the algebra of linear endomorphisms
$\End_\FF(A)$ generated by the left and right multiplications by elements of $A$:
\[
\cM(A)\bydef \alg\langle L_x,R_x\mid x\in A\rangle,
\]
where $L_x\colon y \mapsto xy$, $R_x\colon y\mapsto yx$. 

In order to study the multiplication algebra of a Vidinli algebra, we need some notation. 

Let $W$ be
a vector space over a field $\FF$, $w$ an element of $W$ and $f$ an element of the dual vector space
$W^*$. Denote by $wf$ (or $(f(.)w$) the linear map $u\mapsto f(u)w$. Note that the composition
$w_1f_1\circ w_2f_2$ equals $f_1(w_2)w_1f_2$. Also, for subsets $X\subseteq W$ and 
$Y\subseteq W^*$, $XY$ stands for the linear span of $\{ wf\mid w\in X,\, f\in Y\}$. This is a 
subalgebra of $\End_\FF(W)$. Note that $WW^*$ is the subalgebra of $\End_\FF(W)$ consisting of the 
linear endomorphisms of finite rank of $W$. If the dimension of $W$ is finite, then $WW^*$ equals
$\End_\FF(W)$, but this is not so if the dimension is infinite.

\begin{theorem}\label{th:M(A)}
Let $A$ be a Vidinli algebra over a field $\FF$ of characteristic not $2$ with norm $q$, and let $B$ its 
bilinear form (see \eqref{eq:B}). Then the following assertions hold:
\begin{alphaenumerate}
\item 
The multiplication algebra $\cM(A)$ is contained in $\FF\id +A\bigl(B(A,.)+B(.,A)\bigr)$.

\item 
The subspace $(\rad B)\bigl(B(A,.)+B(.,A)\bigr)$ is an ideal of $\cM(A)$ whose square is trivial.

\item 
If the dimension of $A$ is finite, then $A\bigl(B(A,.)+B(.,A)\bigr)$ equals the annihilator in
$\End_\FF(A)$ of $\rad B$.

\item 
If either $B$ is symmetric and $\dim  A/\!\rad B\geq 3$ or $B$ is not symmetric, then
\[
\cM(A)=\FF\id +A\bigl(B(A,.)+B(.,A)\bigr).
\] 
That is, the containment in item \textup{(a)} is an equality.

\item 
If $B$ is symmetric and $\dim  A/\!\rad B=1$, then
\[
\cM(A)=\FF \id \oplus (\rad B)B(A,.)=\FF\id\oplus (\rad B)B(1,.).
\]

\item If $B$ is symmetric and $\dim  A/\!\rad B=2$, 
pick $x\in V$ with $A=\FF 1\oplus\FF x\oplus\rad B$, then
\[
\cM(A)=\Bigl(\FF\id + \FF \sigma_{1,x}+\FF \pi_x\Bigr)\oplus (\rad B)B(A,.),
\]
where $\pi_x\colon A\rightarrow A$ is the projection onto $\FF 1\oplus \FF x$: $\pi_x(1)=1$, 
$\pi_x(x)=x$ and $\pi_x(\rad B)=0$; while $\sigma_{1,x}=q(1,.)x-q(x,.)1$.
\end{alphaenumerate}
\end{theorem}

\begin{proof}
For any two elements $x,y\in A$, write $\sigma_{x,y}$ for the linear operator $q(x,.)y-q(y,.)x$.

\noindent (a)\quad For any $x\in V$, the linear map $L_x+R_x$ is given by
\[
L_x+R_x=\begin{cases} 1\mapsto 2x=q(1,1)x,\\
                  y\in V \mapsto -q(x,y)1, \end{cases}
\]
so that we have 
\begin{equation}\label{eq:LxRx}
L_x+R_x=\sigma_{1,x}.
\end{equation}
Also we have 
\begin{equation}\label{eq:Lx-Rx}
L_x-R_x=\omega(x,.)1.
\end{equation}
 Hence we get
\[
\begin{split}
\cM(A)&=\alg\langle L_x,R_x\mid x\in A\rangle\\
 &=\FF\id +\alg\langle L_x+R_x,L_x-R_x\mid x\in V\rangle\\
 &=\FF\id +\alg\langle \sigma_{1,x},\omega(x,.)1\mid x\in V\rangle\\
 &\subseteq \FF\id+ A\bigl(q(A,.)+\omega(A,.)\bigr)\\
 &=\FF\id + A\bigl(B(A,.)+B(.,A)\bigr),
\end{split}
\]
where it has been used that $q(x,y)=B(x,y)+B(y,x)$, while $\omega(x,y)=-B(x,y)+B(y,x)$ because
of \eqref{eq:B}.

\smallskip

\noindent (b)\quad For $x,y\in V$ we compute
\begin{equation}\label{eq:sxy}
\begin{split}
[\sigma_{1,x},\sigma_{1,y}]&=\sigma_{\sigma_{1,x}(1),y}+\sigma_{1,\sigma_{1,x}(y)}\\
 &=2\sigma_{x,y}-q(x,y)\sigma_{1,1}=2\sigma_{x,y}.
\end{split}
\end{equation}
This gives
\begin{equation}\label{eq:sAAMA}
\sigma_{A,A}\subseteq \cM(A).
\end{equation}

On the other hand, $\rad B$ is the intersection $(\rad q)\cap (\rad\omega)$ and, as 
$\sigma_{x,y}=-q(y,.)x$ for $x\in\rad B$ and $y\in A$, we get that $(\rad B)q(A,.)$ is contained
in $\cM(A)$. Also, for $x,y\in A$, $L_y\circ \omega(x,.)1=\omega(x,.)y$, and hence we also have
\begin{equation}\label{eq:omegaAAMA}
A\omega(A,.)\subseteq \cM(A).
\end{equation}
This gives, in particular, $(\rad B)\omega(A,.)\subseteq \cM(A)$ and, as a consequence, we get
\[
(\rad B)\bigl(B(A,.)+B(.,A)\bigr)=(\rad B)\bigl(q(A,.)+\omega(A,.)\bigr)\subseteq \cM(A).
\]

Clearly $(\rad B)\bigl(B(A,.)+B(.,A)\bigr)$ is an ideal of $A\bigl(B(A,.)+B(.,A)\bigr)$, and  the
composition $A\bigl(B(A,.)+B(.,A)\bigr)\circ (\rad B)\bigl(B(A,.)+B(.,A)\bigr)$ is trivial, because
$B(A,\rad B)=0=B(\rad B,A)$. Using (a), assertion (b) follows.

\smallskip

\noindent (c)\quad This assertion follows from the following simple fact: If $W$ is a finite 
dimensional vector space and $Y$ is a subspace of the dual $W^*$,
denote by $Y^\circ$ the subspace $\{x\in W\mid f(x)=0\ \forall f\in Y\}$. Then $WY$ is the
annihilator of $Y^\circ$ in $\End_\FF(W)$. 

\smallskip

\noindent (d)\quad Equation \eqref{eq:omegaAAMA} gives $A\omega(A,.)\subseteq\cM(A)$. Besides,
for any $x,y,z\in A$,
\[
\begin{split}
\bigl(\omega(x,.)y\bigr)\circ\sigma_{x,z}
  &=\bigl(\omega(x,.)y\bigr)\circ\bigl(q(x,.)z-q(z,.)x\bigr)\\
 &=\omega(x,z)q(x,.)y,
\end{split}
\]
and this shows that, if $\omega\neq 0$, 
 $q(x,.)y$ belongs to $\cM(A)$ for any $x,y\in A$ with $\omega(x,.)\neq 0$. But the set 
$\{x\in A\mid \omega(x,.)\neq 0\}$ spans $A$, so $q(A,.)A$ is contained in $\cM(A)$ and
\[
A\bigl(B(A,.)+B(.,A)\bigr)=A\bigl(q(A,.)+\omega(A,.)\bigr)\subseteq \cM(A),
\]
thus proving (d) if $\omega\neq 0$ or, equivalently, if $B$ is not symmetric.

On the other hand, if $B$ is symmetric, so that $B=\frac{1}{2}q(.,.)$, and 
$\dim A/\!\rad B=\dim A/\!\rad q\geq 3$, then for any nonisotropic and orthogonal elements $x,y,z\in A$,
we compute
\[
\begin{split}
\sigma_{x,y}\circ\sigma_{x,z}
 &=\bigl(q(x,.)y-q(y,.)x\bigr)\circ\bigl(q(x,.)z-q(z,.)x\bigr)\\
 &=-q(x,x)q(z,.)y\quad\text{as $q(x,y)=q(x,z)=0$.}
\end{split}
\]
This shows that $q(z,.)y$ belongs to $\cM(A)$ for any orthogonal nonisotropic vectors $y,z\in A$.

Besides, for any nonisotropic $x\in A$, take nonisotropic $y,z$ orthogonal to $x$ with $q(y,z)\neq 0$.
(This is always possible because of our dimension assumption.)
As $q(y,.)x\circ q(x,.)z=q(y,z)q(x,.)x$ we conclude that $q(x,.)x$ also belongs to $\cM(A)$. But we
also have that $q(x,.)y$ belongs to $\cM(A)$ for any
nonisotropic $y$ orthogonal to $x$. Since the nonisotropic vectors span the whole vector space, it follows
that $Aq(x,.)$ belongs to $\cM(A)$ for any nonisotropic $x$, and then also that 
$Aq(A,.)=AB(A,.)=AB(.,A)$ is contained in $\cM(A)$.

\smallskip

\noindent (e)\quad If $B$ is symmetric (i.e., $\omega=0$) and $\dim A/\!\rad B=1$, then $\rad B=V$,
so $V^2=0$ by Theorem \ref{th:rad}, $A$ is commutative, and $L_1=R_1=\id$, while 
$L_x=R_x=\frac{1}{2}q(1,.)x$. Hence we get
\begin{multline*}
\cM(A)=\FF\id+\alg\langle L_x\mid x\in V\rangle \\
 =\FF\id +Vq(1,.)=\FF\id\oplus (\rad B)B(1,.)=\FF\id\oplus (\rad B)B(A,.).
\end{multline*}

\smallskip

\noindent (f)\quad Finally, if $\omega=0$ and $\dim A/\!\rad B=2$, then $B=\frac{1}{2}q(.,.)$. Pick
$x\in V$ such that $A=\FF 1\oplus\FF x\oplus \rad q$. The arguments used to prove (a) and (b) show
that $\cM(A)=\FF\id +\alg\langle \sigma_{y,z}\mid y,z\in A\rangle$. But 
$\sigma_{A,A}=\FF\sigma_{1,x}\oplus\sigma_{A,\rad q}=\FF\sigma_{1,x}\oplus (\rad q)q(A,.)$.
By (b) $(\rad q)q(A,.)$ is an ideal of $\cM(A)$, and hence we have
$\cM(A)=\FF\id+\alg\langle \sigma_{1,x}\rangle+(\rad q)q(A,.)$. Besides, we have 
$\sigma_{1,x}^2=-4q(x)\pi_x$,
and $\pi_x\circ\sigma_{1,x}=\sigma_{1,x}\circ\pi_x=\sigma_{1,x}$, and assertion (f) follows.
\end{proof}

\medskip

Let us consider now the Lie multiplication algebra $\cL(A)$ of a Vidinli algebra $A$, that is, the Lie subalgebra
of $\End_\FF(A)^-\defby\frgl(A)$ generated by the left and right multiplications by elements of $A$.

Note that there is a well defined trace on the subalgebra $AA^*$ of $\End_\FF(A)$, even if
the dimension of $A$ is infinite, given by $\trace(af)=f(a)$ for any $a\in A$ and $f\in A^*$. This coincides
with the usual trace if the dimension of $A$ is finite, in which case we have the equality 
$AA^*=\End_\FF(A)$.

\begin{theorem}\label{th:LM(A)}
Let $A$ be a Vidinli algebra over a field $\FF$ of characteristic not $2$ with norm $q$, and let $B$ its 
bilinear form (see \eqref{eq:B}). Then the following assertions hold:
\begin{alphaenumerate}
\item 
If $B$ is symmetric, then $\cL(A)=\FF\id +\sigma_{A,A}$.

\item 
If $B$ is not symmetric, then 
$\cL(A)=\FF\id+\{\varphi\in A\bigl(B(A,.)+B(.,A)\bigr)\mid \trace(\varphi)=0\}$.
\end{alphaenumerate}
\end{theorem}

\begin{proof}
If $B$ is symmetric, so that $\omega=0$, then $A$ is commutative and Equations \eqref{eq:LxRx}
 and \eqref{eq:sxy}, together with the fact that $\sigma_{A,A}$ is closed under
commutators, prove assertion (a).

\smallskip

Assume now that $B$ is not symmetric, so that $\omega\neq 0$. Equations \eqref{eq:LxRx},
\eqref{eq:Lx-Rx}, and \eqref{eq:sxy}, show that 
$\cL(A)=\FF\id +\alg_{\text{Lie}}\langle \sigma_{x,y}, \omega(x,.)1\mid x,y\in A\rangle$, 
where $\alg_{\text{Lie}}$
denotes the Lie subalgebra generated by its arguments. Note that the trace of $\sigma_{x,y}$ and of
$\omega(x,.)1$ is always $0$, so we get
\begin{equation}\label{eq:LABtr}
\begin{split}
\cL(A)&\subseteq \FF\id +\{\varphi\in A\bigl(B(A,.)+B(.,A)\bigr)\mid \trace(\varphi)=0\}\\
 &=\FF \id + \{\varphi\in A\bigl(q(A,.)+\omega(A,.)\bigr)\mid \trace(\varphi)=0\}.
\end{split}
\end{equation}

Let $\pi$ be the orthogonal projection onto $\FF 1$ relative to $q$, that is, $\pi=\frac{1}{2}q(1,.)1$, and
note that its trace is $1$.

For $x,y,z\in A$ compute
\begin{equation}\label{eq:sxywz1}
[\sigma_{x,y},\omega(z,.)1]
=\omega(z,.)\sigma_{x,y}(1)-\omega(z,y)q(x,.)1+\omega(z,x)q(y,.)1\in\cL(A).
\end{equation}
In case $x=1$ and $y,z\in V$, this gives
\[
2\omega(z,.)y-\omega(z,y)q(1,.)1=2\bigl(\omega(z,.)y-\omega(z,y)\pi\bigr)\in\cL(A),
\]
and as $\omega(1,.)=0$ and $\omega(y,.)1\in\cL(A)$ for any $y\in A$, 
the above equation is valid for any $y,z\in A$. In other words, we get
\begin{equation}\label{eq:wxypi}
\omega(x,.)y-\omega(x,y)\pi=\omega(x,.)y-\trace\bigl(\omega(x,.)y\bigr)\pi\in\cL(A).
\end{equation}
Now \eqref{eq:sxywz1}, with $x,y,z\in V$ such that $\omega(x,y)=1$ and $y=z$, gives $q(y,.)1\in\cL(A)$.
But $V$ is spanned by the elements $y$ with $\omega(A,y)\neq 0$, so we get 
\begin{equation}\label{eq:qV1}
q(V,.)1\in\cL(A).
\end{equation}
Since we  know that $\sigma_{1,x}=q(1,.)x-q(x,.)1$ belongs to $\cL(A)$, we also get 
\begin{equation}\label{eq:Vq1}
Vq(1,.)\in\cL(A).
\end{equation}
Besides, for any $x,y\in V$, we have
\[
[\sigma_{1,x},q(y,.)1]=q(y,.)\sigma_{1,x}(1)-q(y,x)q(1,.)1
 =2q(y,.)x-2q(y,x)\pi\in\cL(A),
\]
which, together with \eqref{eq:qV1} and \eqref{eq:Vq1} gives, for any $x,y\in A$,
\begin{equation}\label{eq:qxypi}
q(x,.)y-q(x,y)\pi=q(x,.)y-\trace\bigl(q(x,.)y\bigr)\pi\in\cL(A).
\end{equation}
Finally, if $\varphi\in A\bigl(B(A,.)+B(.,A)\bigr)=A\bigl(q(A,.)+\omega(A,.)\bigr)$, there are $m,n\in\ZZ_{\geq 0}$ and 
elements $x_i,y_i$, $i=1,\ldots,m$, and $x_j',y_j'$, $j=1,\ldots,n$, such that
\[
\varphi=\sum_{i=1}^m q(x_i,.)y_i + \sum_{j=1}^n \omega(x_j',.)y_j'.
\]
If, moreover, the trace of $\varphi$ is trivial, then 
\[
0=\trace(\varphi)=\sum_{i=1}^m q(x_i,y_i) + \sum_{j=1}^n\omega(x_j',y_j'),
\] 
so we get
\[
\varphi=\varphi-\trace(\varphi)\pi= 
\sum_{i=1}^m \Bigl(q(x_i,.)y_i-q(x_i,y_i)\pi\Bigr) 
  + \sum_{j=1}^n \Bigl(\omega(x_j',.)y_j'-\omega(x_j',y_j')\pi\Bigr),
\]
and hence, because of \eqref{eq:wxypi} and \eqref{eq:qxypi}, we conclude that $\varphi$ belongs
to $\cL(A)$.
\end{proof}


\bigskip

\section{Vidinli algebras over fields of characteristic  two}\label{se:char2}

Assume throughout this section that the characteristic of our ground field $\FF$ is $2$. It turns out
that the class of Vidinli algebras is very large in this situation, as it becomes clear already from
Example \ref{ex:char2}, and can be described in terms of a bilinear form and an anticommutative
multiplication.

To begin with, let us prove that, if the dimension is at least $3$, the unity $1$ lies in the radical of
the polar form of the norm. 

\begin{proposition}\label{pr:qA10}
Let $A$ be a Vidinli algebra over a field $\FF$ of characteristic $2$ with norm $q$. Then either 
$\dim_\FF A\leq 2$ or $q(A,1)=0$.
\end{proposition}

\begin{proof}
If the dimension is at least $3$, take elements $x,y$ so that $1$, $x$, and $y$ are linearly independent.
Equations \eqref{eq:conic} and \eqref{eq:omega} give:
\[
xy+yx=q(x,1)y+q(y,1)x-q(x,y)1,\qquad [x,y]=xy+yx\in\FF1.
\]
As $1$, $x$ and $y$ are linearly independent we get $q(x,1)=0$. But the elements $x$ that are
linearly independent to $1$ span the whole algebra $A$, so we conclude that $q(A,1)=0$.
\end{proof}

\begin{remark}\label{re:qA10}
Any conic algebra $A$ with norm $q$ over a field $\FF$ of characteristic $2$ such that $q(A,1)=0$ 
is a Vidinli algebra, because \eqref{eq:conic} gives $x^2=q(x)1$ for any $x\in A$,
and hence, by linearization, we get $[x,y]=xy+yx=q(x,y)1\in\FF 1$.

Therefore, apart from the Vidinli algebras of dimension $1$ or $2$, the remaining Vidinli algebras
are just the conic algebras satisfying that the unity lies in the radical of the polar form of the norm.
\end{remark}

The only Vidinli algebra of dimension $1$ is, up to isomorphism, the ground field. Moreover, any
Vidinli algebra of dimension $2$ is necessarily commutative, so the condition $[A,A]\subseteq \FF 1$
is superfluous, and these algebras are just the two-dimensional conic algebras. Let $A$ be such an
algebra and pick an element $x\in A\setminus \FF 1$. We get the following possibilities:
\begin{itemize}
\item 
If $q(1,x)=0$, then $x^2=q(x)1$ and either
\begin{itemize}
\item[$\bullet$] $q(x)\in \FF^2$, so the element $\varepsilon =x-\sqrt{q(x)}1$ satisfies $\varepsilon^2=0$, and
$A=\FF 1\oplus \FF\varepsilon$ is just the algebra of dual numbers; or
\item[$\bullet$] $q(x)\not\in\FF^2$, and then $A$ is a purely inseparable field extension of $\FF$ of degree $2$.
\end{itemize}

\item If $q(1,x)\neq 0$, then the polynomial $X^2-q(x,1)X+q(x)1$ is separable so either
\begin{itemize}
\item[$\bullet$] $A$ is isomorphic to $\FF\times\FF$ if this polynomial has roots in $\FF$; or
\item[$\bullet$] $A$ is a separable field extension of $\FF$ of degree $2$, otherwise.
\end{itemize}
\end{itemize}

\smallskip

Hence, in what follows, assume that $A$ is a Vidinli algebra of dimension at least $3$. 
Proposition \ref{pr:qA10} gives $q(A,1)=0$, and hence \eqref{eq:conic} gives  $x^2=q(x)1$ for any 
$x\in A$. 

Let $V$ be a subspace of $A$ complementary to $\FF 1$. Note that there is no canonical way of 
choosing this subspace. For any $u,v\in V$, the product $uv$ belongs to $A=\FF 1\oplus V$, 
so there is a bilinear form
$\varphi\colon V\times V\rightarrow \FF$ and a bilinear map (multiplication) 
$*\colon V\times V\rightarrow V$ such that
\begin{equation}\label{eq:uvphi*}
uv=\varphi(u,v)1+u*v.
\end{equation}
As $u^2=q(u)1$, we obtain $q(u)=\varphi(u,u)$ and $u*u=0$ for any $u\in V$. It follows that,
for any  $\alpha\in \FF$ and $u\in V$, $q(\alpha 1+u)=\alpha^2+q(u)=\alpha^2+\varphi(u,u)$.

\medskip

Let us introduce some notation. Let $(V,*)$ be an anticommutative algebra over $\FF$ endowed with
a bilinear form $\varphi\colon V\times V\rightarrow \FF$, then $\cA(V,*,\varphi)$ will denote the
unital algebra defined on the direct sum $\FF 1\oplus V$, with multiplication given by
\begin{equation}\label{eq:AV*phi}
(\alpha 1+u)(\beta 1+v)=\bigl(\alpha \beta +\varphi(u,v)\bigr)1+\bigl(\alpha v+\beta u+u*v\bigr),
\end{equation}
for $\alpha,\beta\in\FF$ and $u,v\in V$.

\begin{remark}\label{re:conicAV*phi}
In a sense, the algebras $\cA(V,*,\varphi)$ are the natural analogues in characteristic $2$ of the
conic algebras in characteristic not $2$, with Equation \eqref{eq:AV*phi} being completely
analogous to \cite[Eq.~(2)]{QuadraticAlternative}. However, the decomposition $A=\FF 1\oplus V$ is not intrinsic in characteristic $2$.
\end{remark}

Note that $\cA(V,*,0)$ is just the unitization of $(V,*)$ (see Example \ref{ex:char2}). Also, the
algebra $\cA(V,*,\varphi)$ is commutative if and only if $\varphi$ is symmetric. (Recall that our
ground field has characteristic $2$ here, so any anticommutative product is commutative.)

We need some extra bit of notation. Given an anticommutative algebra $(V,*)$ endowed with a bilinear
form as above, and given any linear form $f\in V^*$, denote by $*^f$ the new anticommutative 
multiplication on $V$ defined by
\[
u*^f v=u*v+f(u)v+f(v)u,
\]
and the new bilinear form $\varphi^f$ given by
\[
\varphi^f(u,v)=\varphi(u,v)+f(u)f(v)+f(u*v),
\]
for any $u,v\in V$.

\medskip

Our last result describes the Vidinli algebras of dimension at least $3$ over fields of characteristic $2$.

\begin{theorem}\label{th:AV*phi}
Let $A$ be a Vidinli algebra over a field $\FF$ of characteristic $2$ of dimension at least $3$. Then 
$A$ is isomorphic to the algebra $\cA(V,*,\varphi)$ for some anticommutative algebra $(V,*)$
endowed with a bilinear form $\varphi$.

Conversely, given any anticommutative algebra $(V,*)$ endowed with a bilinear form $\varphi$,
the algebra $\cA(V,*,\varphi)$ is a Vidinli algebra, with norm given by 
$q(\alpha 1+u)=\alpha^2+\varphi(u,u)$, for any $\alpha\in \FF$ and $u\in V$.

Moreover, given two anticommutative algebras $(V,*)$ and $(W,\diamond)$, endowed with
respective bilinear forms $\varphi$ and $\psi$, the Vidinli algebras $\cA(V,*,\varphi)$ and
$\cA(W,\diamond,\psi)$ are isomorphic if and only if there is a linear form $f\in V^*$ such
that the triples $(V,*^f,\varphi^f)$ and $(W,\diamond,\psi)$ are isomorphic, that is,
there is a linear isomorphism $\Phi\colon V\rightarrow W$ such that
\[
\Phi(u*^f v)=\Phi(u)\diamond \Phi(v),\qquad \varphi^f(u,v)=\psi\bigl(\Phi(u),\Phi(v)\bigr),
\]
for any $u,v\in V$.
\end{theorem}

\begin{proof}
The arguments above show that any Vidinli algebra is isomorphic to the algebra $\cA(V,*,\varphi)$, where
$V$ is any subspace complementing $\FF 1$, and $*$ and $\varphi$ are determined by
the projections on $V$ and on $\FF 1$ of the multiplication of elements in $V$ as in \eqref{eq:uvphi*}.

Conversely, for any anticommutative algebra $(V,*)$ endowed with a bilinear form $\varphi$,
any element $x=\alpha 1+u\in A=\cA(V,*,\varphi)$ satisfies $x^2-q(x,1)x+q(x)1=0$, where
$q(x)$ is the quadratic form given by $q(x)=\alpha^2+\varphi(u,u)$, 
because we have $q(x,1)=\alpha q(1,1)=2\alpha=0$ and $x^2=(\alpha^2 +\varphi(u,u))1=q(x)1$.
Hence $\cA(V,*,\varphi)$ is conic, relative to the quadratic form $q$, and as $q(1,A)=0$, 
it is a Vidinli algebra
by Remark \ref{re:qA10}.

Finally, if $\Upsilon\colon\cA(V,*,\varphi)\rightarrow \cA(W,\diamond,\psi)$ is an isomorphism of algebras,
$\Upsilon$ takes the unity $1$ to $1$, and there exists a linear form $f\in V^*$ and a linear
isomorphism $\Phi\colon V\rightarrow W$ such that
$\Upsilon(u)=f(u)1+\Phi(u)$ for any $u\in V$. For any $u,v\in V$, straightforward computations give:
\[
\Upsilon(uv)=\Upsilon\bigl(\varphi(u,v)1+u*v\bigr)
   =\bigl(\varphi(u,v)+f(u*v)\bigr)1+\Phi(u*v),
\]
and
\begin{multline*}
\Upsilon(u)\Upsilon(v)=\bigl(f(u)1+\Phi(u)\bigr)\bigl(f(v)1+\Phi(v)\bigr)\\
 =\bigl(f(u)f(v)+\psi(\Phi(u),\Phi(v))\bigr)1 + \bigl(f(u)\Phi(v)+f(v)\Phi(u)+\Phi(u)\diamond\Phi(v)\bigr),
\end{multline*}
thus proving the equalities
\[
\varphi(u,v)+f(u*v)=f(u)f(v)+\psi\bigl(\Phi(u),\Phi(v)\bigr),
\]
which is equivalent to $\varphi^f(u,v)=\psi\bigl(\Phi(u),\Phi(v)\bigr)$, and
\[
\Phi(u*v)=f(u)\Phi(v)+f(v)\Phi(u)+\Phi(u)\diamond\Phi(v),
\]
which is equivalent to $\Phi(u*^f v)=\Phi(u)\diamond\Phi(v)$. Therefore, $\Phi$ is an isomorphism
from $(V,*^f,\varphi^f)$ to $(W,\diamond,\psi)$. 

And the arguments can be reversed to show that if $\Phi$ is such an isomorphism, the the linear
map $\Upsilon$ given by $\Upsilon(1)=1$ and $\Upsilon(u)=f(u)1+\Phi(u)$, for $u\in V$, is
an isomorphism of algebras $\cA(V,*,\varphi)\rightarrow \cA(W,\diamond,\psi)$.
\end{proof}

\smallskip

Our last result shows a curious behavior.

\begin{proposition}\label{pr:nucleus2}
Let $A$ be a Vidinli algebra over a field $\FF$ of characteristic $2$. Then the associative nucleus $N(A)$
coincides with the center $Z(A)$. Moreover, either the center $Z(A)$ equals
$\FF 1$, or the algebra $A$ is commutative.
\end{proposition}
\begin{proof}
If $\dim A\leq 2$, then $A$ is commutative and associative, so that $Z(A)=A=N(A)$. 

Assume, in what follows, that the dimension of $A$ is at least $3$, so that $A$ is, up to isomorphism,
an algebra $\cA(V,*,\varphi)$. For any element $u\in N(A)\cap V$ and any $v\in V$, we have
\begin{multline*}
0=(u,v,u)=(uv)u-u(vu)\\
=\bigl(\varphi(u*v,u)1-\varphi(u,v*u)\bigr)1+\bigl((u*v)*u-u*(v*u)+\varphi(u,v)u
-\varphi(v,u)u\bigr),
\end{multline*}
and, since $(u*v)*u=u*(v*u)$ by anticommutativity, we conclude that $\varphi(u,v)=\varphi(v,u)$, and
hence $u$ lies in the commutative center. It follows that the associative nucleus $N(A)$ is contained
in the commutative center $K(A)$, and hence the center equals the associative nucleus.

Now, for $u,v,w\in V$, their associator is
\[
(u,v,w)=\bigl(\varphi(u*v,w)-\varphi(u,v*w)\bigr)1
   +\bigl((u*v)*w-u*(v*w)+\varphi(u,v)w-\varphi(v,w)u\bigr),
\]
and hence, if $u$ lies in $N(A)$, we get $\varphi(u*v,w)=\varphi(u,v*w)$ for any $v,w\in V$. In the same
vein we check $\varphi(w,v*u)=\varphi(w*v,u)$. Using that the characteristic is $2$, that $*$ is
anticommutative (and hence commutative!), and that $u$ is in the center, this gives,
\[
\varphi(u*v,w)=\varphi(u,v*w)=\varphi(v*w,u)=\varphi(w*v,u)=\varphi(w,v*u)=\varphi(w,u*v).
\]
We conclude that $u*v$ lies in the commutative center, and so does $uv$. In other words, we get
$uA\subseteq K(A)$. But then, for any $x,y\in A$, using that $u$ is in the center, we have
$u[x,y]=[ux,y]=0$, so that $u[A,A]=0$. Now, $[A,A]$ is contained in $\FF 1$, so we conclude
that either $u=0$ or $[A,A]=0$. Hence either $N(A)\cap V=0$ or $A$ is commutative, as required.
\end{proof}

Following Remark \ref{re:conicAV*phi}, the  algebras of the form $\cA(V,*,\varphi)$ that are commutative,
that is, with $\varphi$ symmetric, are the analogues in characteristic $2$ of the conic algebras in
characteristic not $2$ such that the linear map $x\mapsto q(1,x)1-x$ is an involution. (See, e.g., 
\cite[Proposition 2.1(i)]{QuadraticAlternative}.)


\end{document}